\documentclass[11pt]{amsproc}
\usepackage{amssymb}
\usepackage{amsfonts}

\theoremstyle{plain}
\newtheorem{acknowledgement}{Acknowledgement}

\newtheorem{observation}{Observation}

\newtheorem{theorem}{Theorem}
\numberwithin{equation}{section}

\begin{document}
\title[Facial Thue choice index via entropy compression]{On the facial Thue choice index via entropy compression}
\author{Jakub Przyby\l o}
\address{AGH University of Science and Technology, al. A. Mickiewicza
30, 30-059 Krakow, Poland}
\email{przybylo@wms.mat.agh.edu.pl}

\begin{abstract}
A sequence is \emph{nonrepetitive} if it contains no identical consecutive subsequences.
An edge colouring of a path is \emph{nonrepetitive} if the sequence of colours of its consecutive edges is nonrepetitive.
By the celebrated construction of Thue, it is possible to generate nonrepetitive edge colourings for arbitrarily long paths
using only three colours. A recent generalization of this concept implies that we may
obtain such colourings even if we are forced to choose edge colours
from any sequence of lists of size $4$ (while sufficiency of lists of size $3$ remains an open problem).
As an extension of these basic ideas, Havet, Jendrol', Sot\'ak and \v{S}krabul'\'akov\'a proved that for each plane graph,
$8$ colours are sufficient to provide an edge colouring so that every facial path is nonrepetitively coloured.
In this paper we prove that the same is possible from lists, provided that these have size at least $12$.
We thus improve the previous bound of $291$ (proved by means of the Lov\'asz Local Lemma).
Our approach is based on the Moser-Tardos entropy-compression method and its recent extensions by Grytczuk, Kozik and Micek,
and by Dujmovi\'c, Joret, Kozik and Wood.
\end{abstract}

\maketitle

\section{Introduction}
Given any sequence $S$, we call a subsequence $a_1 a_2\ldots a_h a_{h+1}\ldots a_{2h}$ of its consecutive terms a
\emph{repetition} of size $h$ if $a_{i}=a_{h+i}$ for $i=1,\ldots,h$. The sequence $S$ is called \emph{nonrepetitive} if it
contains no repetition of any size. In 1906 Thue~\cite{Thue1} proved that using just three symbols, e.g, 1,2,3, one may produce
arbitrarily long sequences without repetitions, the result the more striking if one realizes that
any attempt of achieving the same over a 2-element alphabet
is condemned to failure already at the fourth term.

Suppose now that each position in the sequence to be constructed is assigned its own alphabet, i.e.,
we are given a sequence $(L_i)_{i=1}^n$ of lists, each of size, say, $k$, and we aim at choosing
elements from these lists
so that the sequence $a_1a_2\ldots a_n$ (with $a_i\in L_i$) obtained contains no repetitions.
It turns out that we are always able to achieve this goal for any sequence length $n$ regardless of the choice of lists,
provided that these are large enough. While the construction of Thue, though very clever, is quite elementary, and is based
on a series of simple substitutions extending any nonrepetitive sequence into a
longer one, then the proofs of the list version of Thue's problem were (until recently)
highly nontrivial and based on variations of the Lov\'asz Local Lemma.
The first result of this type is a special case of a more general theorem due to Alon et al.~\cite{Alonetal} (with $k=\lfloor8e^{16}\rfloor$).
This was followed ba a series of improvements,
see~\cite{GrytczukParis}, \cite{GrytczukIJMMS}, \cite{Sudakov}, \cite{HarantJendrol}. Finally, using so called Lefthanded Local Lemma due to Pegden~\cite{Pegden}, Grytczuk,
Przyby\l o and Zhu~\cite{GrytczukPrzybyloZhu} proved that lists of size $4$ are sufficient
(while the question of wether the lists of size $3$ would suffice remains wide open).
This result was recently rediscovered in the paper of Grytczuk, Kozik and Micek~\cite{GrytczukKozikMicek} with a beautiful and inspiring argument,
which in fact comes down to pure double counting.
Following in Thue's footsteps, their proof not only confirms the existence of the desired sequences but also provides
a quite efficient construction of these by (roughly) the following
simple random procedure.
Choose consecutive elements of the sequence randomly (uniformly and independently)
from the lists provided as long as some repetition is created (or you are run out of lists).
If a repetition occurs, cancel the choices for the second part of the subsequence making up a repetition and continue the procedure starting from the first list without an element chosen. Such algorithm turned out to produce a desired nonrepetitive sequence from lists of size $4$ with the expected running time linear in $n$, see~\cite{GrytczukKozikMicek} for details and further comments.
This approach was inspired by the Moser-Tardos algorithm~\cite{MoserTardos},
designed for the sake of a constructive proof of the Lov\'asz Local Lemma.
This so called \emph{entropy-compression method} was widely discussed (not only) by combinatorists,
see e.g.~\cite{Tao} by Tao and~\cite{Fortnow} by Fortnow.

The results of Thue have been rediscovered repeatedly, e.g., in the work of Morse~\cite{Morse}, which gave birth to symbolic dynamics.
Many generalizations and applications of this basic problem are related with distinct fields of mathematics and computer science, see~\cite{AlloucheShallit,Bean,CurriePattern,GrytczukMiki,Lothaire1}. One of the most widely studied purely combinatorial extensions
concerns nonrepetitive colourings of graphs. Given a graph $G=(V,E)$, we call its vertex colouring $c$ \emph{nonrepetitive}
if the sequences of colours of all paths in $G$ are nonrepetitive.
The least number of colours required to construct such $c$ is called the \emph{Thue chromatic number} of $G$ and is denoted by $\pi(G)$,
see~\cite{GrytczukIJMMS,GrytczukMiki} for a survey of the known results. The highlights among these belong to Alon et al.~\cite{Alonetal},
who proved that random graphs provide the evidence for existence of graphs with Thue chromatic number of order $\Delta(G)^2/\log \Delta(G)$,
and, on the other hand, used the Local Lemma to settle an upper bound in the form of $\pi(G)\leq C \Delta(G)^2$, where $C$ is a constant.
The series of refinements of this method, \cite{GrytczukIJMMS,GrytczukParis,HarantJendrol,Sudakov}, pushed the constant down to $C=12.92$ (see~\cite{HarantJendrol}).
A recent considerable improvement due to Dujmovi\'c et al.~\cite{Dujmovic} was enabled by a further development
of entropy-compression method along the line proposed in~\cite{GrytczukKozikMicek}.
The algorithm and its analysis contained in~\cite{Dujmovic}, whose application essentially yields $C=1$, is the main inspiration of our approach.
Another advantage of this method is that it remains valid also
in the list setting.
Probably the most intriguing open question within this area concerns now settlement wether the parameter $\pi$ is finite for planar graphs.

We shall consider an alternative extension of Thue's research into the field of plane graphs, focusing on edge colourings this time.
Here, by a \emph{facial path} of a plane graph $G$ we shall mean any its path whose consecutive edges
are also consecutive edges of a boundary walk of some face of $G$.
An edge colouring of $G$ is \emph{facial nonrepetitive} if the colours of its every facial path form a nonrepetitive sequence.
The least number of colours required to construct such colouring is called the \emph{facial Thue chromatic index} of $G$
and denoted by $\pi'_f(G)$. The theorem of Thue thus yields $\pi'_f(P_n)\leq 3$ for all $n$. In~\cite{Havetetal} Havet et al. proved that $\pi'_f(G)\leq 8$ for every plane graph $G$, see \cite{Havetetal,JendrolSkrabulakova,SkrabulakovaThesis} for other results
(and~\cite{Alonetal,BaratCzap,BresarKlavzar} for further combinatorial extensions of the Thue's seminal result).
This problem occurred much more demanding in the list setting. The least integer $k$
such that for every list assignment $L:E\to 2^{\mathbb{N}}$ with $|L(e)|\geq k$ there exists
a facial nonrepetitive edge colouring $c:E\to \mathbb{N}$ where $c(e)\in L(e)$ for every $e\in E$
is called the \emph{facial Thue choice index} of $G$ and is denoted by $\pi'_{fl}(G)$.
For this parameter, a much worse upper bound asserting that $\pi'_{fl}(G)\leq 291$ for every plane graph was settled by means of the Lov\'asz Local Lemma by
Schreyer and \v{S}krabul'\'akov\'a~\cite{SchreyerSkrabulakova}.
In this paper we improve this upper bound using entropy-compression method, and prove
the following.
\begin{theorem}\label{main_theoremJP}
For every plane graph $G$,
$$\pi'_{fl}(G)\leq 12.$$
\end{theorem}
Entire second section constitutes an algorithmic proof of Theorem~\ref{main_theoremJP}.
Some supplementary comments and remarks are contained in section three.

\section{The Algorithm}
Let $G=(V,E,F)$ be a plane graph of size $m$, where $F$ is the set of faces of $G$,
and let us consider any fixed set $(L(e))_{e\in E}$ of lists
of \textbf{positive} integers such that $|L(e)|=12$ for every $e\in E$.
Settle any labeling which implies a natural linear ordering of the set of edges, say
$E=\{e_1,e_2,\ldots, e_m\}$.
We shall first describe the basic idea behind the algorithm before its more precise formulation and analysis.

\subsection{Randomized description}
A generic idea of a colouring algorithm we shall examine is the following.
In the first step, choose a colour for $e_1$ randomly and equiprobably from the list $L(e_1)$. In the second step,
choose a random colour for $e_2$ from its assigned list.
If $e_1$ and $e_2$ make up a facial path and their colours are the same,
uncolour $e_2$ and at the beginning of the third step consider $e_2$ once more, otherwise examine $e_3$ in the third step.
In $i$-th step of the algorithm first find an unpainted edge $e_j$ with the smallest index $j$ and randomly choose its colour from $L(e_j)$.
If as a result a repetitive colouring of any facial path $P$ (containing $e_j$) is created,
cancel the choices of colours for all the edges making up this half of $P$
which contains $e_j$ and continue, otherwise move to the next step.
From our argument below it shall be clear that such algorithm almost surely
terminates, i.e., colours all edges
providing a facial nonrepetitive colouring of $G$ from the given lists.
Note however that a repetitively coloured facial path $P$ generated in step $i$ might not be unique.
Before launching the algorithm we must therefore first set \emph{lists of preference} of facial paths for the edges of $G$.
For every $e_j\in E$ we thus fix a linear ordering $P^j_1,P^j_2,P^j_3,\ldots$ of all facial paths of $G$ containing $e_j$,
and in case of more than one of them coloured repetitively (after assigning a random colour to $e_j$), the algorithm will be obliged
to choose the first  of these in this ordering as $P$.
Then such randomized algorithm is unambiguously defined.

\subsection{Deterministic core}
We shall argue that with probability $1$ this algorithm colours all edges of the graph (if it works sufficiently long).
In fact this will be just a straightforward consequence of our crucial perception, that as the number of (admitted) iterations
of the algorithm grows,
then the proportion of its possible executions
after which $G$ is not entirely coloured
to the remaining ones tends to $0$. Note that this is much more than required to just prove Theorem~\ref{main_theoremJP}.

Given integers $a,b$ with $a\leq b$, denote $[a,b]:=\{a,a+1,\ldots,b\}$.
For every list $L(e)$, $e\in E$, fix some linear ordering of its $12$ colours.
Suppose that we admit a limited number of $t\geq 1$ iterations of our algorithm (where $t$ will be required to be `sufficiently large' later on).
Note that the only random `ingredient' in our algorithm `recipe' comes down to deciding
(at most) $t$ times of which colour from a given list should be chosen.
Though we do not know in advance which lists are going to be analysed
in respective iterations, we may carry out all random experiments involved
prior to launching the algorithm due to the linear orderings introduced in the lists.
To this end, it is sufficient
to `throw $t$ times with a 12-sided symmetrical die', and write the outcomes,
each of whom corresponds to the position of a colour in a 12-element list,
in the form of one of $12^t$ possible sequences $(p_1,p_2,\ldots,p_t)\in [1,12]^t$
(each popping up with probability $1/12^t$).
With such vector chosen on input, the algorithm alters to a deterministic one (cf. Algorithm 1 below).

For the analysis of such algorithm, we shall require it
to return a \emph{record} of some information on its operation (cf. with `$R$' in Algorithm 1 below).
Without loss of generality, we may assume that $G$ is connected.
Though this is not indispensable, it simplifies notations.
Let us then fix one of the two possible (cyclical) orientations for a boundary walk of every face of $G$.
These might be settled by simply `writing down' a (cyclical) sequence of labels of the consecutive edges on a boundary walk of every face.
Then every edge $e\in E$ appears exactly two times in such orientations of all faces of $G$.
Let us label such \emph{appearances} of $e$ (arbitrarily) with $1$ and $2$.
Consider an edge $e\in E$ and a facial path $P$ of length $2h$ containing $e$ in $G$.
We may unambiguously write $P$ as a sequence of its consecutive edges $e_{i_1}e_{i_2}\ldots e_{i_{2h}}$
so that $e=e_{i_q}$ with $q\in [h+1,2h]$ (i.e., so that $e$ is closer to $e_{i_{2h}}$ than to $e_{i_{1}}$ along this path).
Let then $v(e,P)$ be a vector of the form $(h,q,a,o)$ with $h$ and $q$ uniquely determined above and the two last terms defined as follows.
According to a definition of a facial path, the string $e_{i_1}e_{i_2}\ldots e_{i_{2h}}$ or $e_{i_{2h}}\ldots e_{i_2} e_{i_1}$
must occur (once or twice) as labels of consecutive edges in the orientation of some face $\alpha$ containing $e$.
Choose any of the occurrences of such string. Since it contains exactly one edge $e$, by making this choice we
also determine one of the two appearances of $e$ in the faces orientations. Let $a$ denote the label of this appearance, $a\in\{1,2\}$.
Then let $o=1$ if the edges of $P$ appear in the order $e_{i_1}e_{i_2}\ldots e_{i_{2h}}$ in the chosen piece of the orientation of $\alpha$
(i.e., the chosen orientation $e_{i_1}e_{i_2}\ldots e_{i_{2h}}$
of $P$  so that $e$ is closer to the end $e_{2h}$ is `consistent with' the orientation of $\alpha$),
or set $o=2$ otherwise. Note that
given any such vector $v=(h,q,a,o)$
and $e\in E$, we may unambiguously `reconstruct'
the unique path $P$ such that $v(e,P)=v$.
This observation is crucial for our further analysis.

Below we finally provide all details of the algorithm in which the number of admitted iterations is limited to $t$.
For convenience we shall consider unpainted edges to be coloured with colour $0$
(ignoring these while searching for repetitions).

\noindent\rule{\textwidth}{.8pt}
\noindent \textbf{Algorithm 1} of colouring of a facial graph $G=(V,E)$ from linearly ordered lists $(L(e))_{e\in E}$ of size $12$
(with fixed lists of preference of facial paths for all edges, fixed orientations of the faces, fixed labelings of appearances of the edges,
and uniquely defined all possible vectors $v(e,P)$), where $c(e)=0$ for every $e\in E$ on input.

\noindent\rule{\textwidth}{.5pt}
\noindent\textbf{Input:} $(p_1,p_2,\ldots,p_t)\in [1,12]^t$\\
\textbf{Output:} a colouring $c:E \to \{0,1,2,\ldots\}$, and\\
\hspace*{1,6cm}  a \emph{record} $R:[1,t]\to \left(\mathbb{N}^2\times \{1,2\}^2\right)\cup\{\emptyset\}$\\
\hspace*{0,3cm} \textbf{for} $i\in [1,t]$ \textbf{do}\\
\hspace*{0,9cm} $R(i)\leftarrow \emptyset$\\
\hspace*{0,3cm} \textbf{end for}\\
\hspace*{0,3cm} $i\leftarrow 1$\\
\hspace*{0,3cm} $J\leftarrow [1,m]$\\
\hspace*{0,3cm} \textbf{while} $i\leq t$ and $J\neq \emptyset$ \textbf{do}\\
\hspace*{1,4cm} $j\leftarrow \min J$\\
\hspace*{1,4cm} $c(e_j)\leftarrow p_i$-th colour in $L(e_j)$\\
\hspace*{1,4cm} \textbf{if} $G$ contains a repetitively coloured path \textbf{then}\\
\hspace*{1,9cm} let $P$ be the first such path in the list of preference for $e_j$,\\
\hspace*{1,9cm} divide $P$ into first half $P_1$ and second half $P_2$ so that $e_j\in E(P_2)$\\
\hspace*{1,9cm} \textbf{for} $e_l\in E(P_2)$ \textbf{do}\\
\hspace*{2,5cm} $c(e_l)\leftarrow 0$\\
\hspace*{2,5cm} $J\leftarrow J\cup\{l\}$\\
\hspace*{1,9cm} \textbf{end for}\\
\hspace*{1,9cm} $R(i)\leftarrow v(e_j,P)$\\
\hspace*{1,4cm} \textbf{else}\\
\hspace*{1,9cm} $J\leftarrow J\smallsetminus\{j\}$\\
\hspace*{1,9cm} $R(i)\leftarrow \emptyset$\\
\hspace*{1,4cm} \textbf{end if}\\
\hspace*{1,4cm} $i\leftarrow i+1$\\
\hspace*{0,3cm} \textbf{end while}\\
\hspace*{0,3cm} \textbf{return} $c, R$\\
\noindent\rule{\textwidth}{.8pt}

\subsection{One-to-one correspondence}
To prove Theorem~\ref{main_theoremJP}
it is sufficient to show that for
at least one input $(p_1,p_2,\ldots,p_t)\in [1,12]^t$, Algorithm 1 terminates with all edges coloured
(i.e., with $J=\emptyset$). Suppose then that it is otherwise, and denote by $\mathcal{R}_t$
the set of records, by $\mathcal{C}_t$ the set of edge colourings $c$ of $G$
(with $c(e)\in L(e)\cup\{0\}$ for $e\in E$), and by $\mathcal{L}_t$ the set of
pairs $(c,R)\in \mathcal{C}_t \times \mathcal{R}_t$
returned by Algorithm 1 for all $(p_1,p_2,\ldots,p_t)\in [1,12]^t$.
Every element $(c,R)\in \mathcal{L}_t$ shall be called a \emph{log}.
Note that
\begin{equation}\label{insignificance_of_c}
|\mathcal{L}_t|\leq |\mathcal{C}_t|\cdot |\mathcal{R}_t|\leq 13^m|\mathcal{R}_t|.
\end{equation}
A typical application of the double-counting rule is based on showing
a one-to-one correspondence between two methods of counting objects in some set.
This is the role of the key observation below.
\begin{observation}\label{one-to-one_observation}
Every log $(c,R)\in \mathcal{L}_t$ is returned for a unique input $(p_1,p_2,\ldots,p_t)\in [1,12]^t$.
\end{observation}

\begin{proof}
We shall prove this fact by induction with respect to $t$.

The assertion is obvious for $t=1$, thus consider any $t>1$ and a fixed log $(c,R)\in \mathcal{L}_t$.
First note that using information provided by the record $R$, we may unambiguously reconstruct
the sequence $J_1,J_2,\ldots, J_t$,
where $J_i$ denotes the set of indices of yet not coloured edges at the beginning of step $i$ of an execution of Algorithm 1 returning $(c,R)$,
simply by analysing consecutively $R(1),R(2),\ldots,R(t)$.
Indeed, reconstruction of $J_{i+1}$ from $J_i$ is obvious if $R(i)=\emptyset$, while for $R(i)=v\neq \emptyset$,
we first identify an edge $e_j$ which was analysed in the $i$-th step of the algorithm by taking $j=\min J_i$,
and then use the fact that given $e_j$ and $v$ we may uniquely determine the path $P$ such that $v(e_j,P)=v$.
(Alternately, the same process of reconstruction might also be supported with an inductive argument.) Now let us consider two cases.

If $R(t)=\emptyset$, then let $j=\min J_t$, hence $e_j$ is the edge analysed by the algorithm in step $t$.
If we then define $R':[1,t-1]\ni i \to R(i)$ and $c':E\to \{0,1,2,\ldots\}$ by setting
$c'(e)=c(e)$ for $e\in E\smallsetminus\{e_j\}$ and $c'(e_j)=0$, then $(c',R')\in \mathcal{L}_{t-1}$.
By induction, there is a unique vector $(p_1,p_2,\ldots,p_{t-1})\in [1,12]^{t-1}$
for which we obtain the record $R'$ and the colouring $c'$ after $t-1$ steps of the algorithm.
Since for $t$ steps admitted, the edge $e_j$ is going to be analysed in the next iteration, and we know that it is supposed to receive the colour $c(e_j)\in L(e_j)$, this also determines the position $p_t\in[1,12]$ of the colour $c(e_j)$ in $L(e_j)$, and the assertion follows.

If on the other hand $R(t)=v\neq\emptyset$, then let again $j=\min J_t$, let $P$ be the unique path with $v(e_j,P)=v$ for $e_j$,
and suppose $e_{i_1}e_{i_2}\ldots e_{i_{2h}}$ are the consecutive edges of $P$ where $j=i_q$ with $q\in [h+1,2h]$.
Then analogously as above, $(c',R')\in \mathcal{L}_{t-1}$ for $R':[1,t-1]\ni i \to R(i)$ and $c':E\to \{0,1,2,\ldots\}$ such that
$c'(e)=c(e)$ for $e\in E\smallsetminus\{e_{i_{h+l}}:l=1,\ldots,h\}$, $c'(e_{i_{h+l}})=c(e_{i_l})$ for $l\in\{1,\ldots,h\}\smallsetminus\{q-h\}$ and $c'(e_j)=0$.
By induction, there is a unique vector $(p_1,p_2,\ldots,p_{t-1})\in [1,12]^{t-1}$
for which we obtain the record $R'$ and the colouring $c'$ after $t-1$ steps of the algorithm.
Since for $t$ steps admitted, the edge $e_j$ is going to be analysed in the next iteration, and we know that
we are supposed to obtain a repetition formed by the edge colour sequence of the unique path $P$ with $v(e_j,P)=R(t)$ for $e_j$,
we thus also know what colour our edge $e_j$ is supposed to obtain. Analogously as above, the assertion follows.
\end{proof}
Now we are ready to carry out the final calculations, which expose the heart of entropy-compression method.
We shall basically prove that if our hypothesis of lack of positive termination of our algorithm regardless
of an input $t$-element string held,
then such random string might have been compressed to a smaller extent than is actually possible.

\subsection{Double counting}
For any numbers $k_1,k_2,\ldots,k_p$ and nonnegative integers $r_1,r_2,\ldots,r_p$,
denote by $k_1^{r_1}k_2^{r_2}\ldots k_p^{r_p}$ the sequence
$$(\underbrace{k_1\ldots,k_1}_{r_1},\underbrace{k_2\ldots,k_2}_{r_2},\ldots,\underbrace{k_p,\ldots,k_p}_{r_p}),$$
where we shall usually write $k_l$ instead of $k_l^1$.
Let further, for any two number sequences $S_1=(b_1,\ldots,b_p)$ and $S_2=(d_1,\ldots,d_r)$, $S_1\oplus S_2$
denote their \emph{concatenation}, i.e.,
$S_1\oplus S_2 = (b_1,\ldots,b_p,d_1,\ldots,d_r)$.
For a given vector $v=(h,q,s,o)$ of positive integers, let then $I(v)=1(-1)^h$,
and set $I(\emptyset)=1$ (i.e., the sequence $(1)$).
For any $R\in \mathcal{R}_t$, let $\mathbb{I}'_R=I(R(1))\oplus I(R(2))\oplus\ldots\oplus I(R(t))$.
In other words, $\mathbb{I}'_R$ is a sequence of $1$'s and $-1$'s constructed from $I(R(1))=(1)$
by repeating the following procedure subsequently
for $i=2,3,\ldots,t$: place $1$ at the end of the constructed sequence $\mathbb{I}'_R$ if $R(i)=\emptyset$,
or place $1$ followed by $h$ of $-1$'s at the end of the constructed sequence if $R(i)=(h,q,s,o)$.
The most important feature of such $\mathbb{I}'_R$ is that the sum of the elements of any its prefix
is nonnegative, since the number of the coloured edges cannot be negative at any stage of the algorithm
(it is positive in fact, since we always have at least one edge coloured).
Consequently, since each of $t$ iterations adds a single $1$ to this sequence, its length never exceeds $2t$.
Thus let us finally define by $\mathbb{I}_R$ the sequence of length exactly $2t$ constructed of $\mathbb{I}'_R$
by placing a respective number of $1$'s at its end. We introduce this definition for the sake of convenience,
so that all sequences $\mathbb{I}_R$ generated by the records $R\in \mathcal{R}_t$ have the same length.
Denote the set of all these sequences by $\mathcal{I}_{2t}$.
For a positive integer $n$, let $\mathcal{J}_n$ be the set of all $n$-element sequences consisting of $1$'s and $-1$'s,
hence $|\mathcal{J}_n|=2^n$ and $\mathcal{I}_{2t}\subset\mathcal{J}_{2t}$ for every $t$.
Set $\mathcal{J}=\bigcup_{n\in\mathbb{N}}\mathcal{J}_n$
and let us define a function $g:\mathcal{J}\to \mathbb{N}$ such that
for any sequence (with at least one `$-1$')
$$\mathbb{I}=1^{z_1}(-1)^{h_1}1^{z_2}(-1)^{h_2}\ldots(-1)^{h_r}1^{z_{r+1}}\in \mathcal{J}_n$$
with $n,r\geq 1$, $h_i>0$ for $i=1,\ldots,r$, $z_i>0$ for $i=2,\ldots,r$ and $z_1,z_{r+1}\geq 0$,
$$g(\mathbb{I}):=4h_1\cdot 4h_2\cdot\ldots\cdot 4h_r,$$
and set $g(1^n)=1$.
For any $\mathbb{I}\in \mathcal{I}_{2t}$, let $\mathcal{R}_{t,\mathbb{I}}$ denote the set of elements
$R\in\mathcal{R}_t$ such that $\mathbb{I}_R=\mathbb{I}$.
Then $|\mathcal{R}_{t,\mathbb{I}}|\leq g(\mathbb{I})$.
It follows from the fact that for a fixed $h$, there are exactly $4h$
vectors of the form $(h,q,a,o)$ with $q\in[h+1,2h]$ and $a,o\in\{1,2\}$.
We thus have:
\begin{equation}\label{R_t_vs_g_I}
|\mathcal{R}_t|=\sum_{\mathbb{I}\in \mathcal{I}_{2t}}|\mathcal{R}_{t,\mathbb{I}}|
\leq \sum_{\mathbb{I}\in \mathcal{I}_{2t}} g(\mathbb{I})
\leq \sum_{\mathbb{I}\in \mathcal{J}_{2t}} g(\mathbb{I}).
\end{equation}
For any positive integer $n$, denote:
\begin{equation}\label{a_n_notion}
a_n = \sum_{\mathbb{I}\in \mathcal{J}_{n}} g(\mathbb{I}),
\end{equation}
hence $|\mathcal{R}_{t}|\leq a_{2t}$.
Then the following recurrence relation holds (for $n\geq 2$):
$$a_n=a_{n-1}+4a_{n-2}+4\cdot 2a_{n-3}+\ldots + 4(n-2)a_1+4(n-1)+4n,$$
where the factor $a_{n-1}$ corresponds to the sequences $\mathbb{I}\in \mathcal{J}_{n}$ with `$1$' at the end,
the factor $4a_{n-2}$ - to the ones with the end of the form `$\ldots 1,-1$',
the factor $4\cdot 2a_{n-3}$ - to the ones with the end of the form `$\ldots 1(-1)^2$',$\ldots$,
the factor $4(n-2)a_1$ - to the ones with the end of the form `$\ldots 1(-1)^{n-2}$',
the factor $4(n-1)$ - to the sequence $1(-1)^{n-1}$, and
the factor $4n$ - to the sequence $(-1)^{n}$.
Note that $a_1=5$, $a_2=17$, $a_3=57$, and for $n\geq 3$:
$$a_n-a_{n-1}=a_{n-1}+3a_{n-2}+4a_{n-3}+4a_{n-4}+\ldots+4_{a_1}+8,$$
and hence:
$$a_n=2a_{n-1}+3a_{n-2}+4a_{n-3}+4a_{n-4}+\ldots+4_{a_1}+8$$
(where we mean $a_3=2a_2+3a_1+8$). Therefore, for $n\geq 4$:
$$a_n-a_{n-1}=2a_{n-1}+a_{n-2}+a_{n-3},$$
and hence:
$$a_n=3a_{n-1}+a_{n-2}+a_{n-3}.$$
Consequently,
$$a_n=c_0\lambda_0^n+c_1\lambda_1^n+c_2\lambda_2^n,$$
where $c_0,c_1,c_2\in\mathbb{C}$ are some (fixed) constants and $\lambda_0,\lambda_1,\lambda_2$
are the roots of
the characteristic equation
$\lambda^3-3\lambda^2-\lambda-1=0$
in the complex domain $\mathbb{C}$, i.e.,
$\lambda_0\approx 3.383$, $\lambda_1\approx -0.191+0.509i$, $\lambda_2\approx -0.191-0.509i$.
These might be precisely calculated using e.g.
Cardano's formula\footnote{By Cardano's formula, if we set $v_0=\left(2+\frac{2}{3}\sqrt{\frac{11}{3}}\right)^{\frac{1}{3}}$ and $u_0=\frac{4}{3}v_0^{-1}$, then $\lambda_0=v_0+u_0+1$, $\lambda_1=-\frac{1}{2}(v_0+u_0)+1+\frac{\sqrt{3}}{2}(v_0-u_0)i$ and $\lambda_2=\overline{\lambda_1}$.}
or a computer programme.
Since $|\lambda_0|\leq 3.383$ and $|\lambda_1|=|\lambda_2|\leq 0.544$, we thus have that:
\begin{equation}\label{a_n_vs_3_383}
a_n\leq |c_0||\lambda_0|^n+|c_1||\lambda_1|^n+|c_2||\lambda_2|^n\leq C\cdot(3.383)^n
\end{equation}
for sufficiently large $n$ and some positive constant $C$.

By~(\ref{insignificance_of_c}), (\ref{R_t_vs_g_I}), (\ref{a_n_notion}) and~(\ref{a_n_vs_3_383}) we thus obtain:
$$|\mathcal{L}_t|\leq 13^mC\cdot(0.383)^{2t}\leq 13^mC\cdot(11.5)^t<12^t$$
for sufficiently large $t$. On the other hand, there are exactly $12^t$ possible
input vectors $(p_1,p_2,\ldots,p_t)\in [1,12]^t$ for our algorithm.
By Observation~\ref{one-to-one_observation} we thus obtain a contradiction with the hypothesis that
none of the executions of Algorithm 1 provides a required edge colouring,
and therefore the proof of Theorem~\ref{main_theoremJP} is completed.

\section{Concluding remarks}
Note that in fact we have showed that the number of input vectors $(p_1,p_2,\ldots,p_t)\in [1,12]^t$ for
which our (randomized) algorithm does not terminate in $t$ steps is for a fixed graph $G$ equal to $O(11.5^t)$.
Since $|[1,12]^t|=12^t$, this indeed proves that
the proportion of the executions of the algorithm
in which $G$ is not entirely coloured after $t$ steps
to the remaining ones tends to $0$ as $t$ grows.
Therefore, the randomized algorithm terminates with probability $1$.

In our reasoning above, we might have also added $-1$'s instead of $1$'s at the ends of sequences while defining
how $\mathbb{I}_R$ arises from $\mathbb{I}'_R$. With a slightly more careful analysis,
for a fixed $t$ this would essentially lead us to considering exclusively so called Dyck words of length $2t$ instead
of all sequences from $\mathcal{J}_{2t}$, where a \emph{Dyck word} of length $2t$ might be defined
as any sequence consisting of $t$ copies of $1$ and $t$ copies of $-1$ arranged so that
the sum of the elements of its every prefix is nonnegative, see e.g.~\cite{Dujmovic} for some further details.
It is known that the number of Dyck words of length $2t$ is expressed by the $t$-th Catalan number
$C_t=\frac{1}{t+1}{2t \choose t}$, which is bounded from above by $\frac{1}{\sqrt{\pi}t^{3/2}}4^t$.
On the other hand, $|\mathcal{J}_{2t}|=4^t$.
It means that we might have reduced the number of sequences under consideration by
the multiplicative
factor $\sqrt{\pi}t^{3/2}$.
This however seams insignificant in context of comparing crucial for the final result quantities,
expressed by exponential functions such as $11.5^t$ and $12^t$.
Therefore we do not suspect this
more complicated approach to the analysis of our algorithm
to yield an improvement of Theorem~\ref{main_theoremJP} (i.e., a reduction of list lengths from $12$ to $11$ or less).

Recall that in the present paper we decrease the previous result of $\pi'_{fl}(G)\leq 291$ from~\cite{SchreyerSkrabulakova}
by showing $\pi'_{fl}(G)\leq 12$ instead.
We might have also obtained a reduction of $291$ by a straightforward application of the Lov\'asz Local Lemma,
in fact less complex than the one presented in~\cite{SchreyerSkrabulakova}.
A bit more careful calculation of `degrees' in dependency digraph combined with a touch of optimization
yields then $\pi'_{fl}(G)\leq 23$, or maybe a slightly smaller constant.
It seams not plausible however to push it down to $12$ using this method.
This confirms once more that entropy compression might serve as a tool for (sometimes significant) improvements of
a wide range of results proved by means of the Local Lemma, see~\cite{Dujmovic,GrytczukKozikMicek}
for more extensive discussion on advantages and possible future prospects of entropy-compression method.

\begin{acknowledgement}
The research contained in this paper were
partially supported by the Polish Ministry of Science and Higher Education.
\end{acknowledgement}

\end{document}